\documentclass[12pt,leqno]{article}
\usepackage{graphicx,psfrag,amssymb,amsmath,amsthm}
 \usepackage[dvips]{color}
  \usepackage{psfrag,comment}
\usepackage{makeidx}
   \usepackage{eucal}
   \usepackage{graphicx,psfrag,amssymb,amsmath,amsthm, comment}
\usepackage{color}
   \newtheorem{theorem}{Theorem}
\newtheorem{lemma}[theorem]{Lemma}

\newtheorem{proposition}[theorem]{Proposition}
\newtheorem{corollary}[theorem]{Corollary}

\newtheorem{defn}[theorem]{Definition}

\def\beq{\begin{equation}}
\def\eeq{\end{equation}}

\def\dist{{\rm dist}}

\def\supp{{\rm supp}}

\def\bB{\mathbf{B}}

\def\bH{\mathbf{H}}
\def\bI{\mathbf{I}}

\def\bR{\mathbf{R}}

\def\bg{\mathbf{g}}
\def\bh{\mathbf{h}}

\def\bp{\mathbf{p}}

\def\bs{\mathbf{s}}

\def\by{\mathbf{y}}

\def\cA{\mathcal{A}}
\def\cB{\mathcal{B}}
\def\cC{\mathcal{C}}

\def\cF{\mathcal{F}}

\def\cH{\mathcal{H}}

\def\cK{\mathcal{K}}

\def\cM{\mathcal{M}}

\def\cO{\mathcal{O}}

\def\cR{\mathcal{R}}
\def\cS{\mathcal{S}}
\def\cT{\mathcal{T}}
\def\cV{\mathcal{V}}
\def\cW{\mathcal{W}}

\def\IR{{\mathbb R}}
\def\IT{{\mathbb T}}

\def\tM{\tilde{M}}
\def\tQ{\tilde{Q}}
\def\tF{\tilde{F}}

\def\hmu{\hat\mu}

\def\dist{\text{\rm dist}}

\newcommand{\eps}{\varepsilon}

\def\pQ{\partial Q}

\numberwithin{equation}{section}

\begin{document}

\title{Decay of correlations for billiards\\
 with flat points I: channel effect}
\author{ Hong-Kun Zhang$^{1}$ }

\date{\today}

\maketitle
\footnotetext[1]{ Department of Mathematics \& Statistics, University of Massachusetts Amherst, MA 01003; Email: hongkun@math.umass.edu}

\begin{abstract}
 In this paper we constructed a special	 family of semidispersing billiards bounded on a rectangle with a few dispersing scatterers. We assume there exists a pair of flat   points (with zero curvature) on the boundary of these scatterers, whose tangent lines form a channel in the billiard table that is perpendicular to the vertical sides of the rectangle.  The billiard can be induced to a Lorenz gas with  infinite horizon when replacing the rectangle by a torus.
We study the mixing rates of the one-parameter family of the semi-dispersing
billiards and  the Lorenz gas on a torus; and show that the
correlation functions of both  maps decay polynomially.

\end{abstract}
\section{Background and the main results.}

A billiard is a mechanical system in which a point particle moves
in a compact container $Q$ and bounces off its boundary $\pQ$. It
preserves a uniform measure on its phase space, and the
corresponding collision map (generated by the collisions of the
particle with $\pQ$, see below) preserves a natural (and often
unique) absolutely continuous measure on the collision space. The
dynamical behavior of a billiard is determined by the shape of the
boundary $\pQ$, and it may vary greatly from completely regular
(integrable) to strongly chaotic.

In this paper we  consider a planar semidispersing billiard table $Q$ on a rectangle $\bR$. There are a few convex obstacles $\bB=\cup_{i}\bB_i$ inside $\bR$, such that the billiard table is defined as $Q=\bR\setminus \bB$. We assume that the boundaries of these obstacles have positive curvatures except on two opposite points in the table, $\bp_1$ and $\bp_2$, which we call the flat points, as they have zero curvature; we also assume that the  tangent  lines of these two flat points are parallel and perpendicular to a pair of straight boundaries of $\bR$.  When replacing the rectangle $\bR$ by the torus, this becomes a classical Lorenz gas with infinite horizon, and we can say that there is an infinite channel in the unfolding table that bounded by tangent lines of these two flat points.

	  The billiard	 flow $\Phi^t$ is defined on the unit sphere bundle $Q\times	 \mathbf{S}^1$ and preserves the
Liouville measure.
Both Lorentz gas  and the semi-dispersing billiards  have been proven to	 enjoy strong ergodic
properties: their continuous time dynamics and the billiard ball maps
are both completely hyperbolic, ergodic, K-mixing and Bernoulli, see
\cite{CH96, GO, OW98, Sin70,SC87} and the references therein. However, these systems	
have quite different statistical properties depending on the geometric properties of the billiard table.

When all scatterers have positive curvature,	 the Lorentz gases	were proven to have fast mixing rates.
  Exponential mixing rates were obtained by Young \cite{Y98} for finite horizon case and by Chernov \cite{C99} under the condition of infinite horizon.	 On the other hand the semi-dispersing billiards have much	 weaker statistical properties.	
Based upon the methods of Young \cite{Y98}, it was proven that the mixing rates	 are of order $\cO(1/n)$, as $n\to\infty$, see \cite{CZ, CZ3}.  See also the recent papers	by  Chernov, Dolgopyat, Szasz and Varju \cite{CD09,DSzV08a}	 and the references therein for related studies on other statistical properties of these systems.

In the current paper we relax the assumption  of strictly positive curvature on the boundary of dispersing scatterers, by adding finitely many flat points (with zero curvature) on the boundary of dispersing scatterers.
 A family of dispersing billiards with flat points	(as a perturbation of Sinai billiards with finite horizon) were constructed  by Chernov and Zhang \cite{CZ2}. It was proved that the mixing rate varies between $\cO(1/n)$ and exponentially fast depending on the parameter.

Here  we consider  modifications of the billiards considered in \cite{CZ2}, by consider a perturbation of semidispersing billiards on a rectangle, whose induced map can be viewed as  a perturbation for Sinai billiards with infinite horizon.  If we replace the rectangle by the torus,   and view the billiard system in the unbound table with periodic scatterers, thus the associated Lorentz gas has an infinite channel bounded by the tangent lines of these two flat points.

To simplify our analysis,  we assume there is only one convex scatterer $\bB$ located at the center of the table, and 	  assume	there are	 $4$ flat points on the boundary $\partial\bB$. In addition, $\partial\bB$ is symmetric about the vertical and horizontal lines passing through the center of $\partial \bB$. Thus 	 the geometric feature of these special points are essentially identical.	 We fix any $\beta\in (2,\infty)$. Denote $\bp$ as one of the flat point and let $\{(s,z)\}$ be the Cartesian coordinate system	 originated at  $\bp$, then	the part of the boundary $\partial\bB$ that containing $\bp$ can be locally viewed as the graph of	 $z=z(s)$, such that for some	 small $\eps>0$:
\begin{itemize}
\item[(\textbf{h1})] $z(s)=-|s|^{\beta} +\cO(|s|^{\beta+1})$, for any $|s|<\eps$;
\item[(\textbf{h2})] The tangent line at $\bp$ is parallel to one straight side of the rectangle $\bR$.
\end{itemize}
The first assumption (\textbf{h1}) implies that	 these special points are indeed {\it flat points} (with zero curvature) for $\beta>2$. In addition (\textbf{h2}) implies that  any	 trajectory with infinite horizon is only tangent to the scatterer at	these special	 points.
  Fig.\ref{fig1} (a)
describes  a billiard table with flat points for $\beta>2$, and $x_p$ is a vector whose trajectory  has infinite horizon.
\begin{figure}[h]
\centering \psfrag{Q}{\scriptsize$\bR$} \psfrag{p1}{\scriptsize$\bp$}
 \psfrag{x}{\scriptsize$x_p$}\psfrag{1}{\scriptsize$\bB$}
\includegraphics[width=4in]{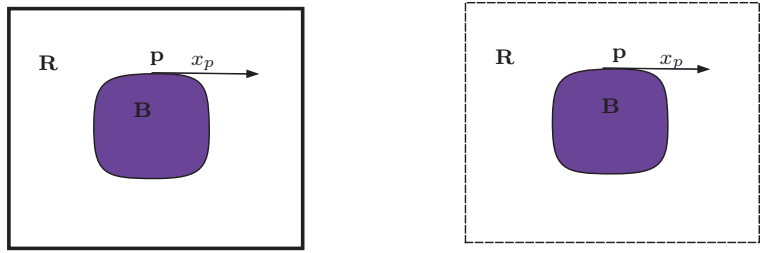}
\renewcommand{\figurename}{Fig.}
\caption{\small\small\small\small\small{(a). Billiards on a rectangle with $4$ flat points; (b). Billiards on a torus with $4$ flat points.\label{fig1}}}
\end{figure}

There is a natural cross section $\cM$ in	 $Q\times	 \mathbf{S}^1$ that contains all postcollision vectors based at the boundary of the table $\partial Q$. The set $\cM=\partial Q \times [-\pi/2, \pi/2]$  is called the collision space.  Any postcollision vector $x\in \cM$ 	 can be represented by $x=(r, \varphi)$, where $r$ is the signed
arclength parameter along $\partial Q$,	 and $\varphi\in [-\pi/2, \pi/2]$ is the angle that $x$	 makes with the inward unit normal vector to the boundary.
The corresponding Poincar\'{e} map (or the billiard map) $\cF: \cM\to \cM$ generated by the
collisions of the particle with $\pQ$ preserves a natural
absolutely continuous measure $\mu$ on the collision space $\cM$.

    For any square-integrable observable  $f,g\in L^2_{\mu}(\cM)$,  {\em
correlations} of $f$ and $g$ are defined
by \beq
   \cC_n(f,g,\cF,\mu) = \int_{\cM} (f\circ \cF^n)\, g\, d\mu -
	\int_{\cM} f\, d\mu    \int_{\cM} g\, d\mu
	   \label{Cn}
\eeq
It is well known that a system $(\cF,\mu)$ is {\em mixing} if and only
if
\beq \label{Cto0}
   \lim_{n\to\infty}
   \cC_n(f,g,\cF,\mu) = 0,
   \qquad \forall  f,g\in L^2_{\mu}(\cM).
\eeq As a result, the rate of mixing of $(\cF,\mu)$ is characterized by the speed of
convergence in (\ref{Cto0}) for smooth enough functions $f$ and $g$.
We will always assume that $f$ and $g$ are bounded, piecewise H\"older continuous or bounded
piecewise H\"older continuous with singularities  coincide with
those of the map $\cF$. We denote $\cW^{u/s}$ as the collection of all unstable/stable manifolds for the billiard map.

For any $\gamma>0$, let $\cH^-(\gamma)$	be the set of all  bounded real-valued  functions $f\in L_{\infty}(\cM,\mu)$ such that  for any $x$ and $y$ lying on one stable manifold $W^s\in \cW^s$,
\beq \label{DHC-} 	|f(x) - f(y)| \leq \|f\|^-_{\gamma} \dist(x,y)^{\gamma},\eeq
with
$$\|f\|^-_{\gamma}\colon = \sup_{W^s\in \cW^s}\sup_{ x, y\in W^s}\frac{|f(x)-f(y)|}{\dist(x,y)^{\gamma}}<\infty.$$
Here  $\dist(x,y)$ denotes the Euclidian distance of $x$ and $y$ in $\cM$,
Similarly, we define $\cH^{+}(\gamma)$ as the set of all bounded, real-valued  functions $g\in L_{\infty}(\cM,\mu)$  such that  for any $x$ and $y$ lying on one unstable manifold $W^u \in \cW^{u}$,
\beq \label{DHC+} 	
   |g(x) - g(y)| \leq \|g\|^+_{\gamma} \dist(x,y)^{\gamma},
\eeq
with
$$
   \|g\|^+_{\gamma}\colon = \sup_{W^u\in \cW^u}\sup_{ x, y\in W^u}\frac{|g(x)-g(y)|}{\dist(x,y)^{\gamma}}<\infty.
$$
For every $f\in \cH^{\pm}(\gamma)$ we define
\beq \label{defCgamma}
   \|f\|^{\pm}_{C^{\gamma}}\colon=\|f\|_{\infty}+\|f\|^{\pm}_{\gamma}.
\eeq

In this paper we  obtain the following  results.
\begin{theorem} \label{TmMain}
For the family of semidispersing  billiards on a rectangle defined as in Fig. \ref{fig1} (a), if $\beta>2$, then the correlations (\ref{Cn}) for the billiard
map $\cF:\cM\to\cM$ and piecewise H\"older continuous functions
$f\in \cH^-(\gamma), g\in\cH^+(\gamma)$,  decay polynomially: \beq \label{main1}
   |\cC_n(f,g,\cF,\mu)|\leq\, C \|f\|^-_{C^{\gamma}}\|g\|^+_{C^{\gamma}}\cdot
   n^{-\eta}
\eeq
where $\eta=1$.
\end{theorem}

\begin{theorem} \label{TmMain1}
For the family of Lorenz gas on a torus defined as in Fig. \ref{fig1} (b), if $\beta>2$, then the correlations (\ref{Cn}) for the billiard
map $\tF:\tM\to\tM$ and
$f\in \cH^-(\gamma), g\in\cH^+(\gamma)$,  decay polynomially: \beq \label{main}
   |\cC_n(f,g,\tF,\tilde\mu)|\leq\, C \|f\|^-_{C^{\gamma}}\|g\|^+_{C^{\gamma}}\cdot
   n^{-\eta}
\eeq
where $\eta=\frac{\beta+2}{\beta-2}$, and $\tilde\mu$ is the conditional of $\mu$ on $\tM$.
\end{theorem}
Moreover, we also have the optimal bound, as obtained by \cite{CVZ}.

\begin{theorem} \label{TmMaino}
For both billiard systems as described above,  consider any
 function $g$ with $ g\in \cH^+$ and $\supp (g)\subset M$.  Then there exists $N=N(g)>0$, such that for any $n>N$,  and any $f\in \cH^-$ supported on  $ M$, the
correlations (\ref{Cn})  decay as:
\beq\label{main3}
\cC_n(f,g) = \mu(R>n)\mu(f)\mu(g)+\|f\|^-_{C^{\gamma}}\|g\|^+_{C^{\gamma}} o(n^{-\eta}).
\eeq
\end{theorem}

\noindent Convention. We use the following notation: $A\sim B$ means that $C^{-1}\leq A/B\leq C$ for some constant $C>1$ . Also, $A=\cO(B)$  means that
$|A|/B<C$  for some constant $C>0$. From now on, we will denote by $C>0$ various constants (depending only on the table) whose exact
values are not important.

\section{General scheme}\label{Sec:3}

 Based upon the methods by Young \cite{Y98} and Markarian \cite{M04},   a general scheme was developed in
\cite{CZ, CZ2, CZ3, CZ09}  on obtaining slow rates of hyperbolic systems with singularities and applied the method on different models.
In the general scheme, first one needs to `localize' spots in the phase space where expansion (contraction) of
tangent vectors slows down. Let $\bar M$ denote the union of all
such spots and $M = \cM\backslash \bar M$. One needs to verify
that the return map $F : M\rightarrow M $ (that avoids all the `bad'
spots) is strongly (uniformly) hyperbolic. One can easily check that it preserves the measure
$\mu_M$ obtained by conditioning $\mu$ on $M$. For any $x\in M$
we call $R (x) = \min\{n\geq  1: \cF^n(x)\in M\}$ the return time function and the  return map $F\colon M \to M$ is defined by
\beq \label{Fdef}
   F(x) = \cF^{R (x)}(x),\,\,\,\,\,\forall x\in M
\eeq

In order to prove the main Theorems,  the
the strategy developed by Chernov and Zhang consists of three steps; they
are fully described in \cite{CZ,CZ3} (as well as applied to several
classes of billiards with slow mixing rates), so we will not bring
up unnecessary details here.\\

\noindent \textbf{(F1)} First, the map $F\colon M\to M$ enjoys
exponential decay of correlations.

\noindent{\small{ More precisely, for any piecewise Holder observables $f,g$ on $M$,  with
$f\in \cH^-(\gamma), g\in\cH^+(\gamma)$, for $\gamma\in (0,1)$, then
$$\int_{M} (f\circ F^n)\, g\, d\mu_M -
    \int_{M} f\, d\mu_M    \int_{M} g\, d\mu_M\leq C \|f\|^-_{C^{\gamma}}\|g\|^+_{C^{\gamma}} \vartheta^n$$
for some uniform constant $\vartheta=\vartheta(\gamma)\in (0,1)$ and $C>0$. }}\medskip

\noindent \textbf{(F2)} Second, the distribution of the return time function $R : M\to [1, \infty)$ satisfies:
  $$\mu( (x\in M\,:\, R(x)\geq n))\sim \frac{1}{n^{1+\eta}},$$
where  $\eta>0$ and $n \geq 1$.

 \medskip

\noindent {\small{It was shown in \cite{CZ} that the same order of upper bound as in Theorem~\ref{TmMain}
follows from (\textbf{F1})-(\textbf{F2}). Also, the proof of (\textbf{F1}) is reduced in
\cite{CZ} to the verification of a one-step expansion condition.}}

The following lemma was proved in \cite{CZ}.
\begin{lemma} \label{LmMain2}
For systems under assumption ({\textbf{F1-F2}}), for the billiard map $\cF:\cM\to\cM$ and any piecewise
H\"older continuous functions $f\in \cH^-(\gamma),g\in\cH^+(\gamma)$ on $\cM$,  the correlations
(\ref{Cn})  decay as
\beq \label{main2}
   |\cC_n(f,g,\cF,\mu)|\leq\,C\|f\|^-_{C^{\gamma}}\|g\|^+_{C^{\gamma}}n^{-\eta}(\ln n)^{1+\eta}\eeq
   for some constant $C>0$.
\end{lemma}

 Now we see that  Theorem~\ref{TmMain} (or Theorem 2) should
follow from (\textbf{F1}) - (\textbf{F2}) for $\eta
=1$ (or $\eta=1+\frac{1}{\beta-1}$, respectively), except the extra logarithmic factor.
To improve the upper bound for the decay rates, one needs to analyze the statistical properties of the return time function.  In \cite{CZ3,CVZ}, the upper bound of decay rates of correlations was improved by dropping the logarithmic factor.

The paper is organized as following.  In  Section 3, we derive properties for the intermediate map -- the Lorenz gas on a torus. We also construct an induced system $(F,M)$ for both maps, by removing collisions on rectangular boundary as well as a small neighborhood of the flat points. The hyperbolicity of the reduced map is proved in Section 4. The assumption \textbf{(F2)} is verified in Section 5, by analyzing the distribution of the return time function. Details of the singular sets is given in Section 6. The regularity of unstable curves is studied in Section 7, including the property of the bounded curvature, distortion bounds, as well as the absolute continuity of the stable holonomy maps. The exponential decay of correlations for the reduced map and \textbf{(F1)} is proved in Section 8, by verifying the One-step expansion estimates, see Lemma \ref{onestep}.    In Section 9, we prove the improved upper bound using method as in \cite{CZ3}.

\section{  Construction of the induced map $(F, M)$}\label{reduced}

\subsection{Properties  of the intermediate Lorentz gas on torus $(\tF,\tM)$.}

In order to construct a good set in the phase space $\cM$, we first consider an intermediate system -- the associated Lorentz gas on torus. More precisely, we replace the rectangle $\IR^2$ by a torus $\IT^2$, then
we have an unbounded table $\tQ$, which is obtained by unfolding  $\mathbb{ T}^2\setminus \bB$.	We denote the corresponding system as $(\tF, \tM)$, with $$\tM=\{x=(r,\varphi)\in \cM\,|\, r\in \partial \bB\},\,\,\,\,\,\,\tF=\cF|_{\tM}$$
Note that the system $(\tF,\tM)$ counts only collisions on the convex scatterer $\bB$. $\tilde\cS_1:=\partial \tM\cup \tF^{-1}\partial \tM$ is the {\it{singular}}  set of $\tF$.
Since the boundary $\partial Q$ is $C^3$,  the map $$\tF: \tM\setminus	 \tilde\cS_1\to \tM\setminus \tF\tilde\cS_1$$ is a local $C^{2}$ diffeomorphism. For any $x\in \tM$, let $\tau(x)$ be the length of the free path between the base points of $x$ and that of the next non-tangential postcollision vector in the unbounded region $\tQ$. A point $x\in \tM$ is said to be an \textit{IH singular point} if  its	 free path is unbounded in $\tQ$, i.e. its forward trajectory never experiences any non--tangential collisions with the boundary of the scatterer (as if  the flat sides of $\IR^2$ are transparent).  In particular we denote by $x_p$ an  IH points based at the flat point $\bp$. Then there exists a	 channel or corridor in $\tQ$ that contains the trajectory of $x_p$.  We assume the $r$-coordinate of $\bp$ is $0$, then $x_p=(0, \pi/2)\in\tM$. By the symmetric   property of the billiard table, it is enough to consider those components  in the vicinity of the IH point $x_p$.

Let $x=(r,\varphi)$ and $\tF x=(r_1,\varphi_1)$.  According to \cite{CM}, for any $x\in \tM \setminus \tF^{-1}\partial \tM$ the differential of $\tF$  is
\begin{align}\label{DTdiff}
D_x \tF=\frac{-1}{\cos\varphi_1}\left(\begin{array}{cc}\tau\cK(r)+\cos\varphi & \tau \\ \tau\cK(r)\cK(r_1)+\cK(r)\cos\varphi_1+\cK(r_1)\cos\varphi & \tau\cK(r_1)+\cos\varphi_1\end{array}\right)
\end{align}
where   $\tau=\tau(x)$ is the distance between the base points of $x$ and $\tF(x)$ in the unbounded table $\tQ$. In addition, for $x\in \tM \cap \tF^{-1}\partial \tM$, $\tF x$ is tangential, we define $\tau(x)$ as the distance between the base of $x$ and  the first non-tangential collision along the forward trajectory of $x$.  By assumption the boundary of $\bB$ is	 $C^3$ smooth, so the map $\tF$ is  $C^2$ smooth	 on each smooth component of $\tM\setminus \tilde \cS_1$.
    The dynamical properties of the intermediate system $(\tF,\tM)$ was studied in \cite{Z12}.  It was shown in \cite{Z12} that  the singularity of the free path $\tau$ divides the region $\tM$ near $x_0$  into countably many connected components, labeled as $$M_n:=(x\in \tM\,:\, \tau(x)=n)$$ for $n\geq 1$.  It was proved  for $n\geq 1$,  an $n$-cell $M_{n}$
is the domain bounded by some smooth curves $\bs_{n}$, $\bs_{n+1}$, $\bs'$
and $\varphi=\pi/2$, such that $\tau$ is smooth on $M_n$.  Then for any $x\in M_n$, the free path $\tau(x)$ is approximately of length $n$ units. To investigate statistical properties of the billiard map, it is crucial to characterize the  distribution of the free path $\tau$. The following lemma was proved in \cite{Z12}.
\begin{lemma}\label{prop2} For $\beta\in [2,\infty)$. Let $x_p=(0,\pi/2)$ be the IH point based at $\bp$.
\begin{itemize}
\item [(1)] Let $\bs\subset \tF^{-1}\partial \tM$ be any smooth curve with equation $\varphi=\varphi(r)$ in the vicinity of $x_p$. Then
 $\bs$	 is a $C^2$ decreasing curve with slope:
$
d\varphi/dr =-\frac{\cos\varphi}{\tau}-\cK$.
\item[(2)] The curve $\bs'$	satisfies
\beq\label{eqs'}
\frac{\pi}{2}-\varphi=\beta r^{\beta-1}+r^{\beta}+\cO(r^{\beta+1}),\,\,\,\forall r\in[0,\eps]
\eeq
\item[(3)] For $n$ large, the curve $\bs_n$ is stretched between $\varphi=\pi/2$ and $\bs'$, with equation satisfying
\beq\label{eqsn}\frac{\pi}{2}-\varphi=\beta r|r|^{\beta-2}+\frac{1}{n}+\cO(n^{-\frac{\beta+1}{\beta-1}}),\,\,\,\,\,\,\,\,\,\forall r\in [-\eps, \eps]
\eeq
 \item[(4)] $M_n$ has length (or $r$-dimension) of order $\cO(n^{-\frac{1}{\beta}})$,	 width (or $\varphi$-dimension ) of order $\cO(n^{-2})$ and $\mu$-density of order $\cO(n^{\frac{1-\beta}{\beta}})$. In addition all boundary components of $M_n$ have uniformly bounded curvature.
\item[(5)]
There exist $c_2>c_1>0$ that do not depend on $\beta$ such that \beq\label{muMn}c_1n^{-3}+\cO( n^{-3-\frac{1}{\beta}})\leq	 \mu(M_n)\leq c_2 n^{-3}+\cO( n^{-3-\frac{1}{\beta}})\eeq
\item[(6)] There exist positive constants $c_1< c_2$, such that for any $n\geq 1$, if
  $ M_{m}\cap \tF	 M_{n}\neq\emptyset$  then
\beq\label{tran}c_1\sqrt[\beta] {n^{\beta-1}}\leq m\leq c_2
\sqrt[\beta-1] {n^{\beta}}\eeq
\end{itemize}
\end{lemma}

Although the intermediate system omit all collisions on the boundary of the rectangle $\IR^2$, it still has points with zero Lyapunov exponents.   It was shown in \cite{Z12} that if $\beta>2$, the intermediate system $(\tF, \tM)$ is nonuniformly hyperbolic. Let
 \beq\label{A0}A_0=\{x=(r,\varphi)\in \cM\,:\, \cK(x)=0, \cK(\tF x)=0\}.\eeq  The following facts were also proved in \cite{Z12}:
\begin{lemma}\label{A0stable}
For any $\beta\in (2,\infty)$, $A_0$ is a null set on which the Lyapunov exponent is zero. Moreover  $A_0$ is not empty and  any $x\in \tM\setminus A_0$ has nonzero  Lyapunov exponent. \end{lemma}
One can check that $A_0$ is made of   periodic points 	based at those flat points. Indeed for any $n\geq 1$, there exists a unique periodic point  $y_n\in A_0\cap M_n$. We define its {\it{ stable set}}  as
$$w_n^s(y_n):=\{ x\in M_n\,|\, \lim_{m\to\infty} d(\tF^m x, y_n)=0\}.$$
Since the Lyapunov exponents are zero at points in $A_0$, we need to eliminate points in $A_0$ and the union of their stable sets, denoted as $w^s$.

\subsection{Construction of the induced map for both $(\cF,\cM)$ and $(\tF,\tM)$.}
In this subsection, we will construct an induced map $(F,M)$ that works for both systems $(\cF,\cM)$ and $\tF,\tM)$.  Later on, we will investigate the induced map $(F,M)$, and to investigate both maps $(\cF,\tM)$ and $(\cF,\cM)$ simultaneously.

Note that by Lemma \ref{prop2} and \cite{Z12}, the zero curvature line $r=0$ cut a cell $M_m$ in to two parts, one of which has very small measure. We fix a small constant $\varepsilon_0 >0$, and for any $m\geq 1$, we define a
 set  $$U_m=([-\eps_0 m^{-\frac{1}{\beta-1}} , \eps_0 m^{-\frac{1}{\beta-1}}]\times [-\pi/2,
\pi/2] )\cap M_{m}.$$
Here $\varepsilon$ can be chosen small enough such that $M_m\setminus U_m$ still contains two regions with $r$-dimension much larger than that of $U_m$.
 Let \beq\label{MU}
   M = \{ x=(r,\varphi)\in \tM \,:\, r\in \partial \bB, x\in M_m\setminus (U_m \cup w^s),\,\, m\in \mathbb{N}\}),
\eeq Note that $M$ only contains points that based on the convex scatterer $\bB$, and we also remove from each cell $M_m$ the stable set $w^s_m$ and a narrow window $U_m$  that contains $y_m$, see Figure.~\ref{FigSing1}. Let $q_1^m$ denote one of the two points on $\pQ$
that border the base of $U_m$, and by $q_2^m$
the other base point.

 First of all we  define a return time function related to the special set $M$,  such that for any $x\in M$, \beq
	 R(x)=\min\{n>0:\ \cF^{n}(x)\in M\}
	   \label{Rx1}
\eeq
We define  the  map $F: M\to M$ as the first return map to $M$, such that for any $x\in M$,
$$F x=\cF^{R(x)} x$$
We call $$C_m:=(x\in M\,:\, \tau(x)=m)=M_m\setminus U_m=C_m'\cup C_m''$$ as the \emph{$m$-cell} of $F$ in $M$ (note this is not the level sets of $R$). $C_m$ contains two disconnected regions from $M_m$, denoted as $C_m'$ and $C_m''$, where $C_m'$ contains almost tangential collisions with angle smaller than $1/m$, while $C_m''$ contains collisions with angle larger than $1/m$.   Moreover, the set  $U_m$ has length ($r$-dimension) $\cO(m^{-\frac{1}{\beta-1}})$, height ($\varphi$-dimension) $\cO(m^{-1})$ and density $\cO(m^{-1})$, thus the measure $\mu(U_m)\sim m^{-3-\frac{1}{\beta-1}}$. By Proposition \ref{prop2}, we know that $\mu(M_m)\sim m^{-3}$. Thus
\beq\label{muCm}\mu(C_m)=\mu(M_m)-\mu(U_m) \sim \mu(M_m)\sim m^{-3}\eeq

Notice for any point $x\in U_m$, and $x_1=\cF x$, we have
\beq\label{cosphi}|r|^{\beta-1}\leq \frac{c}{m},\,\,\,\,\,\,\text{ and }\,\,\,\,\,\,\,\tau(x)\sim m.\eeq
On the other hand let $\cK_m$ be the curvature at the boundary of $U_m$, then for any $x=(r,\varphi)\in C_m$, the curvature at $r$ satisfies
\beq\label{cKCm}\frac{C}{m^{1-\frac{2}{\beta}}}\geq \cK(r)=\beta (\beta-1) r^{\beta-2}\geq \cK_m\geq \frac{c}{m^{1-\frac{1}{\beta-1}}}\eeq
 Clearly $F$ preserves the conditional measure $\hmu$ obtained by restricting $\mu$ on $M$. Next we will check the reduced system for the three conditions (\textbf{F1}) -(\textbf{F3}) proposed in the general scheme.

\section{Hyperbolicity of $(F,M)$}

The key to understanding chaotic
 billiards lies in the study of	 infinitesimal families of trajectories. The basic notion is that of a wave front
along a billiard trajectory. More precisely,  for any $x\in M$, let $V\in \cT_xM$ be a tangent vector. For $\eps>0$ small, let us consider  an infinitesimal curve
$\gamma=\gamma(s)\subset M$, where $s\in (-\eps,\eps)$ is a parameter, such that $\gamma(0)=x$ and $\frac{d}{ds}\gamma(0)=V$
 The trajectories of the points $y\in \gamma$, after leaving $M$, make a  bundle of directed lines in $Q$.
To measure the expanding or contracting  of the wave front, let $\cB=\cB(x)$, which represents the curvature of the orthogonal cross-section of that  wave front at the point $x$ with respect to the  vector $V$.	We say the wave front $\gamma$ is {\it{dispersing}} if $\cB>0$. Similarly, the past trajectories of the points $y\in \gamma$ (before arriving at $M$) make a bundle of directed lines in $Q$ whose curvature
right before the collision with $\partial Q$ at $x$ is denoted by $\cB^-=\cB^-(x)$.
We define $x_1=\tF x=(r_1,\varphi_1)$, then we have \cite{CM}:
\beq\label{cBt}
\cB^-(x_1)=\frac{\cB(x)}{1+\tau(x)\cB(x)}=\frac{1}{\tau(x)+\frac{1}{\cB(x)}}
\,\,\,\text{ and }\,\,\,\,\,{\cB}(x)=\cB^-(x)+\cR(x)\eeq
where $\cR(x)=\tfrac{2\cK}{\cos\varphi}$ is called the {\it{collision parameter}} at $x$.
(\ref{cBt})  implies that	a wavefront that is initially dispersing will stay dispersing. By our construction  on the included map and (\ref{cKCm}), for any $x\in C_m$, with $m\geq 1$, there exist $\tau_{\min}>0$ and $0<\cK_m<\cK_{\max}$ such that for any $x\in M$,
\beq\label{takck}
 \tau(x)\geq \tau_{\min}\,\,\,\,\,\,\text{ and }\,\,\,\,\, \cK_m\leq \cK(x)\leq \cK_{\max}.
\eeq
where $\cK_m\sim m^{-1+\frac{1}{\beta-1}}$ is the minimal curvature at the end point of $U_m$.
Denote by $\cV=d\varphi/dr$ the slope of the tangent line of $W$ at $x$. Then  $\cV$ satisfies
\beq\label{cVx}\cV= \cB^{-}\cos\varphi+\cK(r) = \cB
\cos\varphi-\cK(r)\eeq

 We use the cone method developed by Wojtkowski \cite{Wo85} for establishing hyperbolicity for phase points  $x=(r,\varphi)\in M$. In particular we study stable and unstable wave front. The relations in (\ref{cBt}) and (\ref{cBt}) implies that a dispersing wave front remains bounded away from zero.
  \begin{defn}\label{defn:1}
  The unstable cone $\cC_{x}^u$ contains all tangent vectors based at ${x}$ whose images  generate dispersing wave fronts:
$$\cC^u_x=\{(dr, d\varphi)\in \cT_{x} M\,:\, \cK(r)\leq d\varphi/dr\leq	 \cK(r)+\tau_{\min}^{-1}\}$$
Similarly the stable cones are defined as
 $$\cC^s_x=\{(dr, d\varphi)\in \cT_{x} M\,:\, -\cK(r)\geq d\varphi/dr\geq	 -\cK(r)-\tau_{\min}^{-1}\}$$
 We  say that a smooth curve $W\subset M$ is an unstable
(stable) curve for the reduced system $(F,M)$ if at every point $x \in W$ the tangent line
$\cT_x W$ belongs in the unstable (stable) cone $C^u_x$
($C^s_x$). Furthermore, a curve $W\subset  M$ is an unstable
(resp. stable) manifold for the reduced system $(F,M)$ if $F^{-n}(W)$ is an unstable (resp.
stable) curve for all $n \geq 0$ (resp. $\leq 0$).
  \end{defn}

For any $m\geq 1$, consider a short unstable curve
$W\subset C_m$ with equation $\varphi=\varphi(r)$. Let $x=(q,v)=(r,\varphi)\in W$ and $V=(dr,d\varphi)$ be
a tangent vector at $x$ of $W$.
It
follows from (\ref{cBt})  that
 \beq\label{Bbd}
  \cB^-(x)\leq \tau_{\min}^{-1}, \,\,\,\,\,\,\cK_m\leq  \cos\varphi \cB(x)<\tau_{\min}^{-1}+2\cK_{\max}
 \eeq
Combining with (\ref{DTdiff}), if $\tF x$ stays away from the flat points, then the tangent vector of $FW$ at $x_1=(r_1,\varphi_1)=\tF x$ satisfies
\beq\label{cV1}\cK(r_1)+\frac{\cos\varphi_1}{\tau(x)+\frac{\cos\varphi}{2\cK(r)}}\leq\cV(x_1)\leq \cK(r_1)+\frac{\cos\varphi_1}{\tau(x)}\eeq
This implies that the unstable cone $\cC^u_x$ is strictly invariant under $\tF$, if the base of $\tF x$ stays away from the flat points. Similarly one can check that $D \tF(\cC^s_x)\supset \cC^s_{ \tF x}$.

\begin{lemma}\label{KBbd} Let $m_0$ be a large number. For any $m>m_0$, any unstable curve $W\subset C_m$  such that $\tF^{-1}W$ is also unstable, let  $x=(r,\varphi)\in W$. \\
    (a)  $$\cB^-(x)\sim \tau(x)^{-1},\,\,\,\,\, \cB(x)\sim \frac{\cK}{\cos\varphi}\geq c m^{\frac{1}{\beta}}$$where $c=c(Q)$ is a constant.\\
     \noindent(b) The slope $\cV(x)=d\varphi/dr$ satisfies
     $\cV \sim\cK\sim \cos\varphi^{\frac{\beta-2}{\beta-1}}\geq \cK_m$.
 \end{lemma}
 \begin{proof}

 Since  $m\geq m_0$ is large, the condition $\tF^{-1}W$ is unstable implies that $W$ is almost parallel to the long boundary of $\tF C_n$, for some $n\in [m^{\frac{\beta-1}{\beta}}, m^{\frac{\beta}{\beta-1}}]$. Note that $\tF C_n$ has two part: the smaller region containing points after almost tangential collisions, denoted as $\tF C_n'$; and the larger region with collision angles smaller than $1/n$, denoted as $\tF C_n''$. Thus by (\ref{eqs'}), point $x=(r,\varphi)\in W\cap \tF C_n'$ satisfies
 \beq\label{coscK1}\cos\varphi\sim -\beta |r|^{\beta-1}+\frac{1}{n}\sim n^{-1}\sim \cK^{1+\frac{1}{\beta-2}},\,\,\,\,\,\,\,n\in [m^{\frac{\beta-1}{\beta}}, m]\eeq
 where we have used the definition of $\tF C_n'$, i.e.  $|r|\sim n^{-\frac{1}{\beta-1}}$. Similarly, for $x=(r,\varphi)\in W\cap FC_n''$
 \beq\label{coscK}\cos\varphi\sim \beta |r|^{\beta-1}+\frac{1}{n}\sim |r|^{\beta-1}\sim \cK^{1+\frac{1}{\beta-2}},\,\,\,\,\,\,\,\,n\in [m,m^{\frac{\beta}{\beta-1}}]\eeq

 Moreover this implies that \beq\label{Kcos1}
\frac{\cK}{\cos\varphi} \sim \frac{1}{\cos\varphi^{\frac{1}{\beta-1}}}\geq C m^{\frac{1}{\beta}}
\eeq

Thus the slope of the tangent line of $W$ at $x$ satisfies
 \beq\label{slopeW}
 d\varphi/dr\sim \cK+\frac{\cos\varphi}{\tau}\sim\cK\sim\cos\varphi^{\frac{\beta-2}{\beta-1}}\geq \cK_n
 \eeq
  Combining with (\ref{cVx}) and (\ref{cBt}), we also obtained the following estimates for $\cB^{\pm}$:
  \beq\label{cBKr}
  \cos\varphi\cB(x)\sim \cK,\,\,\,\,\,\,\,\,\text{ and }\,\,\,\,\,\,\,\, \cB^-(x)\sim \frac{1}{\tau}
  \eeq

  \end{proof}

Next we will introduce two metrics on the tangent space $\cT M$. Let $x\in C_m$, for any $m\geq 1$.
\begin{enumerate}
\item[(I)]
The first one is the so-called \emph{p-metric} on  vectors $dx = (dr, d\varphi)$ by
\beq\label{pnorm} |dx|_p = \cos \varphi \, |dr|.\eeq
 Put
$x_1=(r_1,\varphi_1)=\tF x$ and $dx=(dr_1,d\varphi_1)=D \tF(dx)$, for $n\geq 1$.  The expansion factor is \beq
	\frac{|dx_1|_p}{|dx|_p}= 1+\tau(x){\cal
	\cB}(x)\geq 1+\frac{\tau(x)\cK(r)}{\cos\varphi}
			 \label{DTvu1}
\eeq where $\tau(x)\geq c m$ is the collision time for $x$ under $\tF$.

\item[(II)] Now consider the expansion factor in \emph{Euclidean metric} $|dv|^2 =
(dr)^2 + (d\varphi)^2$. Note that
 \beq
\label{penorm}
	\frac{|dx_1|}{|dx|} =
	\frac{|dx_1|_p}{|dx|_p}\,
	\frac{\cos \varphi}{\cos\varphi_{1}}\, \frac{\sqrt{1+(\frac{d\varphi_1}{dr_1})^2}}{\sqrt{1+{(\frac{d\varphi}{dr})^2}}}
\eeq
\end{enumerate}
\section{Distribution of the return time function $R$}

It follows from the definition of $M$, it contains countably many cells $\{C_m\}$ whose boundaries are made of singular curves in $\tilde\cS_1$ for the intermediate system $\tF$. On the other hand, there are new types of singularity curves of $F$ in each
$m-$cell $C_m$ consists of preimages of the vertical boundary components of $U_m$'s.

    For any $x\in C_m$, the value $ R(x)$  corresponding to the number of bounces the billiard
trajectory of the point $x\in C_m$ reflected by $\cF$ in the window $U_m$ before returning to $M$. The singularities of
$R(x)$ occur at points where the number of bounces in the window
 changes from $k$ to $k+1$ or $k-1$, for $k> 1$. Accordingly, the new singular curves in $M$ corresponding to the discontinuities of the return function $R(x)$.

We denote $\{y_m\}$ as the sequence of periodic points with period $m$ that approaching to the IH singular point $x_p$, as $m\to\infty$. Let $w_m^s=w^s_m(y_m)$ be the weak stable manifolds of these periodic points.  For each $m\geq 1$, the singularity set  of the map $F$ in $C_m$ consists of   two
types of infinite sequences of singularity curves $\{s_{m,k}\}$
and $\{s'_{m,k}\}$, approaching to $w^s_m$ from both sides. These curves  correspond to the discontinuities of the function
$R(x)$ in $C_m$.  More precisely, the region above $w^s_m$ but
below $c_m$ consists of points whose trajectories enter the window
and  turn back without reaching  $\gamma_{y_m}$. This region is divided into a sequence of
almost parallel strips by $\{s_{m,k}\}$, $k\in \mathbb{N}$, see
Figure.~\ref{FigSing1}. Denote by $C_{m,k}$ the strip bounded by
$s_{m,k}$, $s_{m,k+1}$ and $\partial U_m$.
 The region blow $w^s_m$ but above $c_m'$
consists of points whose trajectories enter the window and manage
to move through it crossing  $\gamma_{y_m}$. This region is divided into a sequence of almost
parallel strips by $\{s'_{m,k}\}$,  $k\in \mathbb{N}$. Let
$C'_{m,k}$  be the strip bounded by $s'_{m,k}$, $s'_{m,k+1}$ and $\partial U_m$. Both $C_{m,k}$ and $C'_{m,k}$ consist of points
experiencing exactly $k$ collisions with the boundary of $Q$
before returning to $M$.

We also denote $C_{m,0}=C'_{m,0}$ as those points $x\in C_m$, such that $\tF x\in M$, i.e. these points do not enter into the windows $\{U_m. m\geq 1\}$ under iteration of $\tF$. Moreover, for any $n>1$, the level sets of the return time function $R$ can be represented as
\beq\label{Rlevel}(x\in M\,:\, R(x)=n)=\bigcup_{m\geq 1}\bigcup_{k=[n/m]} (C'_{m,k}\cup C_{m,k})\eeq

\begin{figure}[h]
\center \psfrag{M}{$M$}
\psfrag{s1,1}{\scriptsize$s'_{m,k}$}\psfrag{u}{\scriptsize$U_m
$}\psfrag{v}{\scriptsize$U_{m+1}
$}
\psfrag{z}{\scriptsize$y_0$}\psfrag{cm'}{\scriptsize$c_m'$}
\psfrag{s'}{\scriptsize$s_{m,k}'$}\psfrag{cm}{\scriptsize$c_m$}\psfrag{}{\scriptsize$z$}
\psfrag{S0}{\scriptsize$S_0$}\psfrag{1}{\scriptsize$C_m$}
\psfrag{w}{\scriptsize$w^s_m$}\psfrag{S1,1}{\scriptsize$S_{m+1}$}
\psfrag{phi}{\scriptsize$\varphi$}\psfrag{r}{\scriptsize$r$}\psfrag{s2}{\scriptsize$S_m$}
\includegraphics[width=4in,height=2in]{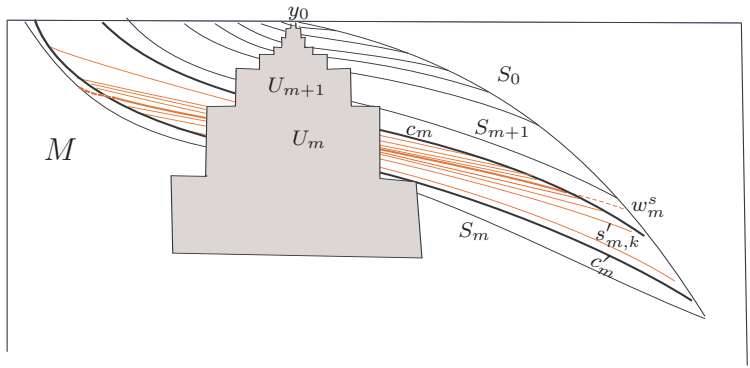}
\renewcommand{\figurename}{Figure.}
\caption{Singularity curves of $F$ in the  vicinity of $C_m$.}\label{FigSing1}
\end{figure}

Fig \ref{FigSing1} shows the structure of these new singular curves in the cell $C_m$.   There are bold steeply decreasing curves $c_m$ terminating on $\bs_{0}$ or $\partial M$, consist of the points in
$C_m$ whose trajectories hit the point $q^m $ of $\pQ$
under the map $\cF$. The stable manifold $w^s_m$ crosses $c_m$,
consists of points whose trajectories converge to $\gamma_{y_m}$.
The dashed part of $w^s_m$  does not enter the window immediately,
but will do so in one iteration.

 In order to determine the rates of the
decay of correlations we need certain quantitative estimates on
the measure of the regions $C_{m,k}$ and $C_{m,k}'$ and on the
factor of expansion of unstable manifolds $W \subset C_{m,k}$ and
$W\subset C_{m,k}'$ under the map $F$. We first state a lemma about the expansion of points in $C_{m,k}$ along all collisions in the region $U_m$.

\begin{proposition} \label{PrAux1} For any $m\geq 1, k\geq 1$, let $W\subset C_{m,k}$ be an unstable curve, and $\Lambda(x)$ be its expansion factor under the map $F=\cF^{mk}$ at $x\in W$.\\
(1) The expansion factor in the Euclidian metric satisfies:
$$\Lambda(x)\sim m^2\cK(x)k^{3+\frac{4}{\beta-2}} \sim m^{1+\frac{1}{\beta-1}} k^{3+\frac{4}{\beta-2}} $$
 \noindent(2) If $W$ is  a vertical curve then $\Lambda(x)\sim  m^{2} k^{3+\frac{4}{\beta-2}}$. Here $C>0$ is constant.
\end{proposition}

The proof of this proposition is rather lengthy, and will be given in the Appendix using a similar
approach as in \cite{CZ2}.
\begin{lemma}\label{muC} For any $m\geq 1$, $k\geq 1$, the set $C_{m,k}$  have measure of
order $\cO(m^{-3-\frac{1}{\beta-1}}k^{-3-\frac{4}{\beta-2}})$.
\end{lemma}

\begin{proof}
Due to the time-reversibility of the billiard dynamics, the
singular curves in $F M$ have a similar structure. Furthermore,  the  short sides of $C_{m,k}$ stretch completely in $U_m$ under $\tF^{m(k-1)}$, while under $\cF^{mk} $ they are transformed into long sides of $F(C_{m,k})$, with length $\sim m^{\frac{1}{1-\beta}}$. Let
$h_{m,k}$ denote the length  of a vertical curve $W$ that stretch completely in $C_{m,k}$.
Thus by Proposition  \ref{PrAux1}, the expansion factor on $W$ is of order $\cO(m^2k^{3+\frac{4}{\beta-2}})$. This implies that
\beq\label{hmk}h_{m,k} \sim\frac{|FW|}{m^2
k^{3+\frac{4}{\beta-2}}}\sim \frac{1}{m^{2+\frac{1}{\beta-1}}
k^{3+\frac{4}{\beta-2}}}.\eeq
It follows from Lemma \ref{prop2} that the length of $C_{m,k}$ is $\cO(m^{-1/\beta})$, and  the density on  $C_{m,k}$
is of order $\cO(m^{-1+\frac{1}{\beta}})$.
 According to  (\ref{hmk}),  the
measures of both $C_{m,k}$ and $C'_{m,k}$ are of order
$\cO(\frac{1}{m^{3+\frac{1}{\beta-1}}k^{3+\frac{4}{\beta-2}}})$.
\end{proof}

It follows from above estimation, we know that
$$\cup_{k\geq 1} \mu(C_{m,k})\sim m^{-3-\frac{1}{\beta-1}}\sim \mu(U_m)$$
Note that from (\ref{muCm}), we have $\mu(C_m)\sim m^{-3}$. This implies that in this sequence $\{C_{m,k}. k\geq 0\}$, only the first set $C_{m,0}$ dominates, as \beq\label{Cm0}\mu(C_{m,0})\sim \mu(C_m)\sim m^{-3}\eeq

Combining with Lemma \ref{muC} and (\ref{Rlevel}),  for any $n\geq 1$, we know that
the set $$(x\in M\,:\, R(x) \geq n)\subset (\cup_{m=n}^{\infty} C_m) \bigcup \left(\cup_{m=1}^{n-1}\sum_{k\geq n/m} C_{m,k}\right)$$
Note that
\begin{align}\label{muMm}
\sum_{m=1}^{n-1}\sum_{k\geq n/m}\mu( C_{m,k})&\sim \sum_{m=1}^{n-1}\sum_{k\geq n/m}\frac{1}{m^{3+\frac{2}{\beta-1}}k^{3+\frac{4}{\beta-2}}}\sim  \frac{1}{n^{2+\frac{1}{\beta-1}}}\end{align}
According to (\ref{Cm0}) we know that
$$\mu(C_{m,0})\sim \mu(R=m)\sim m^{-3}$$
 It is
immediate that for any large $n$, the set $C_{n,0}$ not only dominates in $(R=n)$.  but also  dominates the set $M_n=(x\in \tM\,:\, \tau(x)=n)$. This implies that
 \begin{align}\label{Mpol2}
 \mu&(x\in M:\ R(x)\geq  n)
	\sim   \mu(x\in \tM\,:\, \tau(x)\geq n)\sim \,  \frac{1}{n^2}.\end{align}

This verifies condition (\textbf{F2}) with $\eta=1$ for the system $(\cF,\cM)$.

Moreover, for the intermediate system $(\tF,\tM)$, if we let $\tilde R:M\to \mathbb{N}$ to be the first return time, such that
$$F x=\tF ^{\tilde{R}(x)}(x)$$
Then for any $n\geq 1$,
$$(\tilde R=n)=\cup_{m=1}^{\infty} C_{m,n}$$ Now we use Lemma \ref{muC}, to get
\beq\label{leveltR}\mu(\tilde R\geq n)=\sum_{m\geq 1}\sum_{k\geq n} \mu(C_{m,k})\sim n^{-2-\frac{4}{\beta-2}}\eeq
This implies that $\eta=1+\frac{4}{\beta-2}=\frac{\beta+2}{\beta-2}$ for system $(\tF,\tM)$.

\section{Singularity of the reduced map}
 Let
$$\cS_F=(\tilde\cS_1\cap M)\bigcup \cup_{m\geq 1}\left(\partial U_m \cup F^{-1}\partial U_m\right)$$ For a fixed $m\geq 1$, let $W$ be an unstable curve contained in $ C_m\setminus \cS_F$. It follows from (\ref{cKCm}) that  the curvature $\cK$ is bounded away from zero in  cell $C_m$. According to (\ref{penorm})  the expansion factor at $x\in W$  becomes very large, as $\cos\varphi_1$ gets small. This implies that  the expansion factors along  $W$ maybe highly nonuniform. To overcome this difficulty we divide $M$ into horizontal strips as introduced in \cite{BSC90, BSC91}.  More precisely, one divides $M $ into countably many sections (called
\emph{homogeneity strips}) defined by
$$
	\bH_k=\{(r,\varphi)\in M \colon \pi/2-k^{-2}<\varphi <\pi/2-(k+1)^{-2}\}
$$
and
$$
	\bH_{-k}=\{(r,\varphi)\in M \colon -\pi/2+(k+1)^{-2}<\varphi < -\pi/2+k^{-2}\}
$$
for all $k\geq k_0$ and \beq \label{bbH0}
	\bH_0=\{(r,\varphi)\in M \colon -\pi/2+k_0^{-2}<\varphi <
	\pi/2-k_0^{-2}\},
\eeq here $k_0 \geq 1$ is a fixed (and usually large) constant.

Let $\cS_H$ be the collection of all boundaries of these homogeneity strips. Then $F^{-1} \cS_H$ is a countable union of stable curves that are almost parallel to each other and accumulate on the singular curves in $\tilde S_1$. This implies that there are \textbf{three sequences of singular curves} in each $m$-cell $C_m$: the sequences $\{s_{m,k}\}$ and $s_{m,k}'$ accumulate to the curve $w_m^s$ and  the third one belongs to $F^{-1} \cS_H$ that converges to $\bs_{m+1}$. Define
\beq\label{singular1}\cS_1=\cS_H\cup \cS_F,\,\,\,\,\,\,\,\,\cS_{-1}= F\cS_1\cup \cS_1.\eeq Clearly
$M\setminus \cS_1$ is a dense set in $M$. Whenever we say `the singularity of $F$', we refer to $\cS_1$ although curves in $\cS_H$ are really artificial singular curves for the map.

Since billiards have singularities, if the orbit of $x$ approaches to the singularity set $\cS_1^{ }$ too fast under $F$, then $x$ may not have a stable manifold. For the finite horizon case, one can show that	$m$-a.e. $x\in M$ does have a stable (resp. unstable) manifold. However the situation is much complicated here as we have more accumulated sequences of singular curves. Indeed we will first show that small neighborhood of the singular set has small measure.
\begin{lemma}  For any $\delta>0$, the $\delta-$neighborhood of $\cS_{\pm 1}^{ }$ has measure:
 \beq\label{mu0m}\mu(B_{\delta}(\cS_{\pm 1}^{ }))\leq C  \delta^\frac{2\beta}{3\beta-2}\eeq
  Here $B_{\delta}(\cS_{\pm 1}^{ })=\{x\in M\,:\, d_M(x,\cS_{\pm 1}^{ })\leq \delta\}$ for any $\delta>0$, and $C>0$ is a constant.
\end{lemma}
\begin{proof}
We first need to find the smallest $m_{\delta}$ such that $\cup_{m\geq m_{\delta}} C_m\subset \mu(B_{\delta}(\cS_{ 1}^{ }))$.
According to Lemma \ref{prop2}, the height of $C_m$ is approximately $\cO(m^{-2})$. Thus we get $m_{\delta}=\delta^{-1/2}$. This implies that
$$\mu(\cup_{m\geq m_{\delta} }C_m)=\cO(m_{\delta}^{-2})=\cO(\delta)$$

Next we consider for  $m<m_{\delta}$. In each $C_m$, there are two sequences of cells bounded by singular curves in $\cS_1$: one of which  consists of curves  ${\gamma_k}\subset F^{-1}\{H_k\}$ that approaching $\bs_{m+1}$; and the other is $\{\bs_{m,k}\}$ that approaching $w_m^n$.

 \textbf{Case I.} We assume $W$ to be a vertical curve intersecting a converging sequence  ${\gamma_k}\subset F^{-1}\{H_k\}$ that approaching $\bs_{m+1}$. Then $FW$ is contained in the smaller region $C_m''$ in $C_m$. Let $W_k$ be a vertical curve completely stretch in $F^{-1}(H_k)$. Then  the expansion on $W_k$ is approximately  $\tau/\cos\varphi_1\sim  mk^2$.  Since the $\varphi$-dimension of $FW_k$ is approximately $k^{-3}$. By Lemma \ref{KBbd}, we know that the slope of tangent vector of $FW_k$ is approximately $\cos\varphi_1^{\frac{\beta-2}{\beta-1}}\sim k^{-\frac{2\beta-4}{\beta-1}}$. Thus
           $$|W_k|\sim k^{\frac{-3\beta+1}{\beta-1}}m^{-1}$$
           Let $k_m$ be the smallest integer such that $F^{-1}H_k\cap C_m$ is not empty. As it follows from Lemma \ref{prop2} that the height of $C_m$ is approximately $m^{-2}$, thus
               $$\sum_{k\geq k_m}|W_k|=\sum_{k\geq k_m}k^{\frac{-3\beta+1}{\beta-1}}m^{-1}\sim m^{-2}$$
                   This implies that
                   $$k_m\sim m^{\frac{\beta-1}{2\beta}} $$

           Define $k_m'$ such that
           $$\bigcup_{k\geq k_m'} (F^{-1}H_k\cap C_m)\subset \mu(B_{\delta}(\cS_{\pm 1}^{ }))$$
               Then $k_m'^{\frac{-3\beta+1}{\beta-1}}m^{-1}=\delta$, which implies that $$k_m'=\left(\frac{1}{m\delta}\right)^{\frac{\beta-1}{3\beta-1}}$$
               Since $ H_k$ is contained in $C_m''$, so it has $\varphi$-dimension$\sim k^{-3}$, length $\sim |FW_k|\sim k^{-1-\frac{2}{\beta-1}}$ and density $\sim k^{-2}$.
               Then $$\sum_{k\geq k_m'} \mu(F^{-1}H_k\cap C_m)\sim \sum_{k\geq k_m'} \mu(H_k\cap C''_m)\sim c  k_{m'}^{-5-\frac{2}{\beta-1}}$$
This implies that
$$\sum_{m= 1}^{m_{\delta}} \sum_{k\geq k_m'} \mu(F^{-1}H_k\cap C_m)+\sum_{m\geq m_{\delta}} \mu(C_m)\leq C  \delta$$
\vspace{0.2in}

\textbf{Case II.}  We assume a vertical curve $W$ intersects a sequence $\{\bs_{m,k}\}$ that approaching $w_m^n$. Assume $$\delta=h_{m,k_m''}=m^{-\frac{3\beta-2}{\beta-1}}
k_m''^{-\frac{3\beta-2}{\beta-2}}$$ Then
$$k_m''=\left(\delta^{-1} m^{-\frac{2\beta-1}{\beta-1}}\right)^{\frac{\beta-2}{3\beta-2}}=\frac{ 1}{\delta^\frac{\beta-2}{3\beta-2} m^{\frac{\beta-2}{\beta-1}}} $$
By Lemma \ref{muC}, the measure of these sets in $C_m$ satisfies:
      $$\sum_{k\geq k_m''} \mu(C_{m,k})\leq C m^{-3-\frac{1}{\beta-1}}k_{m''}^{\frac{2\beta}{2-\beta}}=C m^{-3-\frac{1}{\beta-1}}\delta^\frac{2\beta}{3\beta-2} $$

Combining the above facts, we have
\begin{align*}
\mu(B_{\delta}(\cS_{\pm 1}^{ }))&\leq \sum_{m= 1}^{m_{\delta}}\left( \sum_{k\geq k_m'} \mu(F^{-1}H_k\cap C_m)+\sum_{k\geq k_m''} \mu(C_{m,k})\right)+\sum_{m\geq m_{\delta}} \mu(C_m)\leq C  \delta^\frac{2\beta}{3\beta-2}
    \end{align*}

\end{proof}
 Since we have added the boundaries of the homogeneity strips and their preimages in the singular set, the above results implies that for any $\delta>0$, the $\delta-$neighborhood of $\cS_{\pm 1}^{ }$ has measure:
 \beq\label{mu0m1}\mu(B_{\delta}(\cS_{\pm 1}^{ }))=\cO(\delta^{a})\eeq with $a=\frac{2\beta}{3\beta-2}$.
 For any $x\in M$, let $r^{\sigma}_{ }(x)=d_{W^{\sigma}(x)}(x,\partial W^{\sigma}( x))$, where $W^{\sigma}(x)$ is the stable (resp. unstable) manifold that contains $x$, for $\sigma\in \{s,u\}$.

\begin{lemma} For any small $\delta>0$, the set $(r^{s}_{ }(x)<\delta)$ has measure:
\beq\label{rus2}
\mu(r^{s}_{ }(x)<\delta)\leq C\delta^{a}\eeq
\end{lemma}
\begin{proof}  It follows from hyperbolicity there exists $c>0$ such that for any small $\delta>0$,
   $$(r^{s}_{ }(x)<\delta)\subset \cup_{n\geq 0}  F^{-n}_{ }B_{c\delta\Lambda^{-n}_p}(\cS^{ }_{-1}).$$
   We apply the measure $\mu$ and get
   $$\mu(r^{s}_{ }(x)<\delta)\leq \sum_{n\geq 0}  F^{n}_{ }\mu(B_{c\delta\Lambda^{-n}_p}(\cS^{ }_{-1})) \leq \sum_{n\geq 0}  \mu(B_{c\delta\Lambda^{-n}_p}(\cS^{ }_{-1})) \leq C\delta^{a}$$
   for some $C>0$.

 \end{proof}
 The above lemma  implies that almost every point in $M$ has a regular stable (resp. unstable) manifold and there are plenty of reasonable long stable (resp. unstable)  manifolds for $F_{ }$.

\section{Exponential decay rates for the reduced system}

According to the general scheme proposed in Section 2, we first need to check condition (\textbf{F1}), i.e. to prove that the reduced system $(F,M,\hat\mu)$ enjoys exponential decay of correlations.
Here we use a simplified  method to prove exponential decay of
correlations for our reduced billiard map. It is mainly based on
recent results in \cite{ Y98, C99, CZ}.

\begin{lemma}\label{onestep}(\textbf{One-step expansion estimate}) Assume $\beta\in(2, \infty)$, let $W$ be a short unstable curve in $M$ and $\{W_i\}$ be the collection of smooth components in $W$. Then \beq
   \liminf_{\delta_0\to 0}\
  \sup_{W\colon |W|<\delta_0}\sum_{i\geq 1}
\frac{|W_i|}{|FW_i|}<1,
	  \label{step1}
\eeq where the supremum is taken over unstable curves $W\subset M\setminus \cS_1$ and $\Lambda_i$,
$i \geq 1$, denote the minimal local expansion factors of the
connected component $W_i$ under the map $F$.
\end{lemma}

\begin{proof}
Let $W\subset M$ be an unstable curve.  Since the upper bound of (\ref{step1}) only achieves when $W$ intersects one of the accumulating sequences.\\

      \noindent{\bf{Case I.}} First we assume $W$ crosses some $\bs_{m+1}$, the boundary of cell $C_m$, for some $m\geq 1$. Then it must cross a sequence in $ F^{-1}\{\partial \bH_k\}$.
              Define $W_k=W\cap F^{-1}\bH_k$.
               Let $x\in W_k$, and $x_1=Fx$, then  the expansion factor
          $$\frac{|FW_k|}{|W_k|}\sim \Lambda_k(x)\sim \frac{1+\tau\cK(x)}{\cos\varphi_1}\geq c\frac{m k^2}{m^{1-\frac{1}{\beta-1}}}\geq c m^{\frac{1}{\beta-1}} k^2$$
    On the other hand, since for $x=(r,\varphi)\in M_m$, by (\ref{cKCm}), the curvature $\cK(r)\leq  C m^{\frac{2}{\beta}-1}$. Thus we also get the upper bound for $\Lambda_k\leq C m^{\frac{2}{\beta} }k^2$.          Since by Lemma \ref{KBbd}, the slope of $FW_k$ is approximately $\cK(x_1)$, thus $|FW_k|\sim k^{-3}/\cK(x_1)\sim k^{-\frac{\beta+1}{\beta-1}}$, where we have used (\ref{coscK}), as $$\cK(x_1)\sim \cos\varphi_1^{\frac{\beta-2}{\beta-1}}\sim k^{-\frac{2\beta-4}{\beta-1}}$$
                  Thus $|W_k|\geq |FW_k|/\Lambda_k\geq m^{-\frac{2}{\beta}} k^{-2-\frac{\beta+1}{\beta-1}}$. This enable us to find $k_m$ -- the  smallest $k$ such that $F^{-1}\bH_k$ intersects $W$:
                      $$|W|=\sum_{k\geq k_m} |W_k|\geq  m^{-\frac{2}{\beta}} k_m^{-\frac{2\beta}{\beta-1}}$$
                           which implies that $$\frac{1}{k_m}\leq\left(m^{\frac{2}{\beta}}|W| \right)^{\frac{\beta-1}{2\beta}}$$

   This implies that
      \begin{align*}  \sum_{k\geq k_m}\frac{1}{\Lambda_k} &\leq  C m^{\frac{1}{1-\beta}} k_m^{-1}\leq C |W|^{\frac{\beta-1}{2\beta}} m^{-\frac{1}{\beta-1}+\frac{\beta-1}{\beta^2}}  \leq C_1 |W|^{\frac{\beta-1}{2\beta}}\end{align*}
   Thus by taking $|W|$ small, we can make the above sum $<1$.\\

\noindent{\bf{Case II.}} Next we consider the case when $W$ only intersects $\{s_{m,k}\}$. By (\ref{hmk}), we can find $k_m''$ -- the smallest $ k$ such that $W$ intersects $s_{m,k}$:
      $$|W|\sim \sum_{k\geq k_m''}{h_{m,k}m^{1-\frac{1}{\beta-1}}\sim m^{-\frac{\beta+1}{\beta-1}}
k_m''^{-\frac{2\beta}{\beta-2}}}$$
 Then
$$k_m''=\left(|W|^{-1} m^{-\frac{\beta+1}{\beta-1}}\right)^{\frac{\beta-2}{2\beta}}=\frac{ 1}{|W|^\frac{\beta-2}{2\beta} m^{\frac{(\beta+1)(\beta-2)}{(\beta-1)2\beta}}} $$

  Note that $W\subset C_m$, which also implies that $|W|<m^{-2}$, i.e. $m<|W|^{-\frac{1}{2}}$.  By Proposition \ref{PrAux1}  the expansion factor satisfies
    \beq\label{lambdamk}\Lambda_k:=|FW_k|/|W_k|\geq c  m^{1+\frac{1}{\beta-1}} k^{\frac{3\beta-2}{\beta-2}}\eeq
Combining the above facts, we have
      \begin{align*}
      \sum_{k\geq k_m''}
 \Lambda_k^{-1} &\leq C\sum_{k\geq k_m''}\frac{1}{m^{1+\frac{1}{\beta-1}} k^{\frac{3\beta-2}{\beta-2}}}\leq C_1 m^{\frac{1}{\beta-1}}|W|<C_1 \sqrt{|W|}
   \end{align*}
   Thus by taking $|W|$ small, we can make the above sum $<1$.\\

\noindent{\bf{Case III.}} Finally we consider the case when $W$ intersects countably many $C_m$. Let $m_{0}$ be the smallest $m$ such that $W\cap C_m$ is empty. Then $|W|\sim m_0^{\frac{1}{1-\beta}}$. For any $m>m_0$, let $W_{m,k}=C_m\cap W\cap F^{-1}\bH_k$, with $k\geq k_m=m^{\frac{\beta-1}{2\beta}}$ and $\Lambda_{m,k}$ be the expansion factor on $W_{m,k}$. Then by (\ref{lambdamk}).
$$\sum_{m\geq m_0}\sum_{k\geq k_m} \Lambda_{m,k}^{-1}\leq \sum_{m\geq m_0}\sum_{k\geq k_m}\frac{1}{m^{1+\frac{1}{\beta-1}} k^{\frac{3\beta-2}{\beta-2}}}\leq C |W|$$
Again  by taking $|W|$ small, we can make the above sum $<1$.\end{proof}

Given  an unstable curve $W$, a point $x\in W$ and an integer $n\geq 0$, we denote by $r_{n}(x)$ the distance between $F^n x$ and the
boundary of the homogeneous component of $F^n W$ containing $F^n x$. Clearly $r_n(x)$ is a function on $W$ that characterize the size of smooth components of $F^n W$.
We first state  the Growth lemma, proved in \cite{CZ09}, which is key in the analysis hyperbolic systems with singularities. It expresses the fact that the expansion of unstable curves dominates the cutting by singular curves, in a uniform fashion for all sequences. The reason behind this fact is that unstable curves expand at a uniform exponential rate, whereas the cuts accumulate at only finite number of singular points. The following Growth Lemma can be derived directly from Lemma \ref{onestep} -- the one-step expansion estimates, see \cite{CM}, \cite{CZ09} for details.
\begin{lemma}\label{lmmstep} (Growth lemma). There exist uniform constants $C_{\bg},c > 0 $ and $\vartheta\in (0,1)$  such that,
for any unstable curves $W$ and $n\geq 1$ :
$$m_W(r_n(x)<\eps)\leq C_{\bg} (\vartheta^nm_W(r_1(x)<\eps)+c|W|)\eps$$
\end{lemma}
This lemma  implies that for $n$ large enough, such as $n\geq n_W:=|\log  |W|/\log\vartheta |$, one has
\beq\label{growth}
m_W(r_n(x)<\eps)\leq 2C_{\bg} |W|\eps
\eeq
In other words, after a sufficiently long time $n\geq n_W$, the majority
of points in $W$ have their images in homogeneous components of $F^n W$ that are longer than $1/(2C_{\bg})$, and the family of points belonging to shorter ones has a linearly decreasing tail.

In \cite{C99,CZ,CZ09}, the following lemma was proved.
\begin{lemma}\label{lem:1}
If the induced billiard map $F$ satisfies (\ref{step1}). Then there is a horseshoe $\Delta_0\subset M$ such
that \beq\label{Yexp1}
	\mu\bigl(x\in M:\ R(x;F,\Delta_0)>m\bigr)\leq
	\,C\theta^m\quad\quad \forall m\in \mathbb N,
\eeq for some $\theta<1$, where $R(x; F, \Delta_0)$ is the return
time of $x$ to $\Delta_0$ under the map $F$. Thus
 the map $F:M\to M$ enjoys exponential decay of correlations.
\end{lemma}

Hence we conclude that for $\beta\in(2, \infty)$, the return
map $F\colon M\to M$ has exponential mixing rates by above Lemma, thus condition (\textbf{F1}) is verified. Moreover, Lemma \ref{LmMain2} follows with the decay rates given by (\ref{main2}). Thus for both systems $(\cF,\cM)$ and $(\tF,\tM)$, we know the decay rates is of order $\cO(n^{-\eta})$, using Lemma \ref{LmMain2}. Next we will improve the decay rates.

\section{Proof of the main Theorems.}

   A general strategy for estimating the correlation function $\cC_m(f, g,\cF,\mu)$ for
systems with weak hyperbolicity was developed in \cite{CZ, CZ3,CVZ}.

We first prove two results that will be need in the proof of the main Theorems.
\begin{lemma}\label{Dnaa} Let $D_{n}(a)=\{x\in  M\,:\, R( x) \geq n^{1-a}\}$, with $a\in [0,\frac{1}{\beta}]$.
Then for any $n\geq 1$,
$$\mu(D_n(a)|F (R=n))\sim  n^{\frac{\beta a-1}{\beta-1}}$$
In addition the conditional expectation $\mathbb{E}(R(Fx)|x\in M_n)\sim n^{\frac{\beta-1}{\beta}}$.\end{lemma}
It follows from (\ref{Cm0}) that $(R=n)$ is essentially dominated by $C_{n,0}=(R=n)\cap M_n$. Thus the proof of the above lemma directly follows from Lemma 12 in \cite{Z12}, which we will not repeat here.

\begin{proposition}\label{barMm}
There exist $e=\frac{1}{2\beta}$, such that for any large $m$, there exists
$E_m\subset (R=m)$, with $\mu((R=m)\setminus E_m)\leq
m^{-e}\mu(R=m)$, and any $x\in E_m$, $Fx, F^2x, ..., F^{b\ln
m} x$ all belong to cells with index less than $m^{1-e}$.
\end{proposition}

\begin{proof}
It follows from  Lemma \ref{Dnaa},
\beq\label{Rmnm}\sum_{n=m^{1-\frac{1}{2\beta}}}^{\infty}
\mu\left((R=n)|F(R=m)\right)=\cO(m^{-\frac{1}{\beta-1}})\eeq

 Since $C_{m,0}$ dominates $C_m$ and $(R=m)$, we only need to deal with $C_{m,0}$ and its iterations.
  Note that $C_{m,0}$ has dimension $\asymp m^{-2}$ in the unstable direction, dimension $\asymp m^{-\beta}$ in the stable direction, and measure $\mu(C_{m,0})\asymp m^{-3}$.  We first foliate $C_{m,0}$ with unstable curves $W_{{\alpha}}\subset C_{m,0}$ (where $\alpha$ runs through an index set $\cA$). These curves have length $|W_{\alpha}|\asymp m^{-2}$. Let $\nu_m:=\frac{1}{\mu(C_{m,0})}\mu|_{C_{m,0}}$ be the conditional measure of $\mu$ restricted on $C_{m,0}$. Let $\cW=\cup_{\alpha\in \cA} W_{\alpha}$ be the collection of all unstable curves, which foliate the cell $C_{m,0}$. Then we can disintegrate the measure $\nu$ along the leaves $W_{\alpha}$. More precisely,
for any measurable set $A\subset C_{m,0}$,
$$\nu_m(A)=\int_{\cA}\nu_{\alpha}(W_{\alpha}\cap A) \, d\lambda(\alpha)$$
where  $\lambda$ is the probability factor measure on $\cA$.   For each unstable curve $W_{\alpha}\in\cW$, if $F^l W_{\alpha}$  crosses $D_m:=D_m((2\beta)^{-1})$, then $F^l W_{\alpha}$ is cut into pieces by the boundary of cells in $D_m$. Next we use  Lemma \ref{lmmstep}, notice that if $F^l(x)\in C_{n,0}$, for $l=1, \cdots, b\ln m$, $n>m^{1-\frac{1}{2\beta}}$,
then the length of the largest  unstable manifold  is
$\sim m^{-2(1-\frac{1}{2\beta})}=m^{-\frac{2\beta-1}{\beta}}$. So
applying (\ref{growth}) with $\delta=m^{-\frac{2\beta-1}{\beta}}$.   According to the growth lemma \ref{lmmstep}, there exists $\theta\in(0,1)$, such that we have
\beq\label{growthest}F^l_*\nu_m(D_m)\leq c\theta^l  F_*\nu_m(D_m)+C_z m^{-\frac{2\beta-1}{\beta}}.\eeq
  By (\ref{Rmnm}), we know that
  $$F_*\nu_m(D_m)\leq C m^{-\frac{1}{2(\beta-1)}}$$for $\eps=\frac{1}{2\beta}$. Thus by (\ref{growthest}), we get
  $$ F^l_*\nu_m(D_m)\leq c\theta^l  m^{-\frac{1}{2(\beta-1)}}+C_z m^{-\frac{2\beta-1}{\beta}}$$
  Now we sum for $l=1,\cdots, b\ln m$ to get
  $$ \sum_{l=1}^{b\ln m} F^l_*\nu_m(D_m)\leq c\theta^l  m^{-\frac{1}{2(\beta-1)}}+C_z m^{-\frac{2\beta-1}{\beta}}\leq C m^{-\frac{1}{2\beta}}$$ for $b$ large. This also implies that  for any large $m$, there exists
$E_m\subset (R=m)$, with $$\mu((R=m)\setminus E_m)\leq
m^{-\frac{1}{2\beta}}\mu(R=m)$$ and any $x\in E_m$, $Fx, F^2x, ..., F^{b\ln
m} x$ all belong to cells with index less than $m^{1-\frac{1}{2\beta}}$.

\end{proof}

Now we are ready to prove Theorem 1.

		The tower in
$M$ can be easily and naturally extended to $\cM$, thus we get a
the Young's tower with the same base $\Delta_0 \subset M$; and
a.e.\ point $x\in \cM$ again properly returns to $\Delta_0$ under
$\cF$ infinitely many times. Consider the return times to $M$
under $\cF$ for $x\in \cM$. According to Lemma \ref{muC},
   \beq
	\mu(x\in \cM:\ R(x)>n)\sim
	\, \frac{1}{n}\quad\quad \forall n\geq 1
	   \label{Mpol1}
\eeq

 For every $m\geq 1$ and $x\in\cM$ denote
$$
  r(x;m,M)=\#\{1\leq i\leq m:\ \cF^i(x)\in M\}
$$
Let\begin{align*}
	A_m &=\{x\in \cM\colon\ R(x;\cF,\Delta_0)>m\},\\
	B_{m,b} &=\{x\in \cM\colon\ r(x;m,M) > b\ln m\},
\end{align*}
where $b>0$ is a constant to be chosen shortly.

By (\ref{Yexp1}), we know that
	$$
	\mu(A_m \cap B_{m,b})\leq\,C\cdot m\,\theta^{b\ln m}.
$$
Choosing $b=-2\beta/\ln\theta$, then \beq\label{bt}{\rm const}\cdot
m\,\theta^{b\ln m}\leq {\rm const}\cdot m\,
\theta^{-\frac{2}{\ln\theta}\ln m}={\rm const}\cdot m^{-\beta}.\eeq

The set $A_m \setminus B_{m,b}$ consists of points $x\in \cM$
whose images under $m$ iterations of the map $\cF$ return to $M$
at most $b \ln m$ times but never return to the `base' $\Delta_0$
of Young's tower. Our goal is to show that $\mu(A_m \setminus
B_{m,b}) = \cO(m^{-1})$.

Let $I=[n_0, n_1]$ be the longest interval, within $[1, m]$,
between successive returns to $M$. Without loss of generality, we
assume that $m-n_1 \geq n_0$, i.e.\ the leftover interval to the
right of $I$ is at least as long as the one to the left of it
(because the time reversibility of the billiard dynamics allows us
to turn the time backwards).
	  Due to Lemma \ref{barMm}, for large measure of typical points $y \in M_{|I|}$ we have $F^t (y) \in M_{m_t}$
where $m_t$ decreases exponentially fast. So there exists $c>0$,
such that $$ m/2\leq |I|+b|I|^{1-e}\ln |I|\leq c|I|$$ which gives
$|I|\geq \tfrac{m}{2c}$.

Let $G_m=\{x\in A_m\setminus B_{m,b}\,:\, |I|\geq
\tfrac{m}{2c}\}$. Thus it is enough to estimate the size of $G_m$.
Since for any $x\in G_m$, one of its forward images belongs to
$m_{|I|}$ with $ |I|\geq \tfrac{m}{2c}$. Applying the bound
(\ref{muMm}) to the interval $I$ gives
\beq\label{ABslack}
	 \mu(G_m)
	 \leq C\,m\cdot\, m\,\cdot\, m^{-3}=\,C m^{-1}
\eeq
 (the extra factors of $m$ must be included because the
interval $I$ may appear anywhere within the longer interval
$[1,m]$, and the measure $\mu$ is invariant).

 In terms of Young's tower $\Delta$, we
obtain \beq
	\mu(x\in\Delta:\ R(x;\cF,\Delta_0)>m)\leq
	\,C  m^{-1},\quad\quad \forall m\geq 1
	   \label{Ypol3}
\eeq  This completes the proof of the theorem \ref{TmMain}.

\medskip
\medskip

Next we will consider the system $(\tF,\tM)$. Note that by (\ref{leveltR}), we know that measure of level set of the return time function for $\tF$ satisfies:
\beq\label{leveltR1}\mu(\tilde R\geq n)=\sum_{m\geq 1}\sum_{k\geq n} \mu(C_{m,k})\sim n^{-2-\frac{4}{\beta-2}}\eeq
 Since Section 8 have verified (\textbf{F1}) for the system $(\tF,\tM)$, it is enough to verify Proposition \ref{barMm} for  $(\tF,\tM)$ to get the improved bound. Since the estimation is very similar, although with different value of $e$,  we will not repeat here.

\medskip
\medskip

For Theorem 3. It directly follows from results in \cite{CVZ}, once we verify conditions (\textbf{F1})-(\textbf{F2}) and Proposition \ref{barMm}, which is equivalent to condition \textbf{(H2)(b)} in \cite{CVZ}.

\section{Proof of Proposition~\ref{PrAux1}}
\label{secPP} The proof is similar to the proof of Proposition 1
in \cite{CZ2}.  For completeness, we provide a detailed proof using similar notations.

 Assume $W\subset C_{m,k}$ is an unstable curve.  Then the expansion factor (in the Euclidian metric) along $x\in W$ satisfies   \beq\label{verticalLa}
  \Lambda(x) = \prod_{j=0}^{k} \bigl(1 + \tau(x_j)
  \cB(x_j)\bigr)\frac{\cos\varphi_j}{\cos\varphi_{j+1}}\geq \frac{C\tau(x_k)\cK(x_k)}{\cos\varphi_{k}} \prod_{j=1}^{k=1} \bigl(1 + \tau(x_j)
  \cB(x_j)\bigr)\eeq
   where $x_j =(r_j, \varphi_j)= \cF^{jm} (x)$, for $j=0,\cdots,k-1$, with $x_0=x$.
 Note that $\tau(x_j)\sim m,\,\cos\varphi_j\sim m^{-1}$ for $j=1,\cdots,k$.  Denote $\Lambda_{m,k}(x)=\prod_{j=1}^{k-1} \bigl(1 + \tau(x_j)
  \cB(x_j)\bigr)$ as the   expansion factor  in the $p$-metric along all remaining collisions in the series.  Thus
  \beq \label{Lambdan21}
  \Lambda(x) \geq Cm^2 \cK_m\Lambda_{m,k}(x)\geq C m^2\cdot m^{-\frac{\beta}{\beta-1}}\Lambda_{m,k}(x)\geq C m^{1+\frac{1}{\beta-1}} \Lambda_{m,k}(x)\eeq
  Next we will estimate the lower bound for $\Lambda_{m,k}(x)$.

  Note that  $\cB (x_j)$ satisfies the recurrent formula \beq
\label{cBXm}
   \cB(x_{j}) = \frac{2\cK(r_{j})}{\cos\varphi_{j}}
   +\frac{1}{\tau (x_{j-1}) + 1/\cB(x_{j-1})},
\eeq Since $\tau (x_{j-1})$ is the
distance between $x_{j-1}$ and $x_j$ in the unfolding table
$\tilde{Q}$. Clearly, $\tau(x_j)\sim m$, for any $j=1, ...,
k-1$ and  it is easy to compute  \beq \label{cKrm}
	\cK(r_j) \sim \beta(\beta-1)|r_j|^{\beta-2}. \eeq
Now we see that in order to estimate the expansion factor $\Lambda(x)$ given by (\ref{Lambdan21}), it is enough to estimate the $r$-coordinates $\{r_j\}$ along the trajectory of  $x\in C_{m,k}$.


 Let    $x$ be a  point in $C_{m,k}$ (the case $x \in
C'_{m,k}$ is easier and will be treated later). Denote $s_x$ as the distance from $x$ to $\gamma_{y_m}$ in $Q$.
Define a variation flow $\{\psi(s,t): s\in [0, s_x], t\in
(-\infty, \infty)\}$ of $\gamma_{y_m}$ such that $\psi(0,t)$ has
the same trace as $\gamma_{y_m}$ and $\psi(s_x,t)$ has the same
trace as $\gamma_x$. Assume the variation vector field
$J(t)=\frac{\partial \psi}{\partial s}(s_x,t)$ is perpendicular to
the trajectory of $x$, then
$J(t)$ is also called a (generalized) Jacobi field along
$\gamma_{x}$.  Denote by $J_j$ the corresponding
Jacobi vector based at $x_i$. Notice $\dot J=\frac{d J}{dt}$
is also a vector field along $\gamma_{x}$. Correspondingly, we
denote by $v_j=-\dot J_j$, which also represents the angle made by
the orbit of $\cF^{jm}(y_m)$ and that of $x_j=\cF^{jm}(x)$.

 Note that $(\beta r_j^{\beta-1}, 1)$ is
the inward normal vector to $\pQ$ at the point $r_j$. Geometric
considerations (more precisely, the generalized Jacobi equations)
yield the following relations: \beq \label{wxm}
  \begin{split}
  v_{j}-v_{j+1} &= 2\arctan (\beta r^{\beta-1}_{j+1})\\
  J_{j}-J_{j+1} &=m  \tan v_j+2(r_j^{\beta}+r_{j+1}^{\beta})\tan
  v_j.
  \end{split}
\eeq  Also notice
$$r_{j+1}=\frac{J_{j+1}}{\cos\varphi_{j+1}}.$$ Since $C_{m,k}'$ contains points
whose images will be trapped in the window  during the next
$k-1$ iterations of $\cF^m$.    Thus using Taylor expansion we obtain \beq\label{wxm1}
  \begin{split}
	v_{j}-v_{j+1}& = 2\beta r_{j+1}^{\beta-1} - R_{v, j+1}\\
	r_{j}-r_{j+1}& = m^2 v_j + R_{r, j}\end{split}
\eeq where  \beq \label{R1}
  R_{v, j+1}\sim r_{j+1}^{3(\beta-1)}
   > 0
\eeq and \beq \label{R2}
   R_{r, j}\sim mr_j^{\beta}v_j+m v_j^3
   > 0
\eeq (the positivity of $R_{v, j+1}$ and $R_{r, j}$ is
guaranteed by the smallness of $\varepsilon$  and by geometry,
$v_j\ll r_j^{\beta-1}$).

Note that formula (\ref{wxm1})
 differs from  formula (4.8) in \cite{CZ2} only by a constant $m^2/2$. As a result, the proofs in the Appendix of that paper can be directly applied to our model without much changes. In fact we can
get an idea to understand the above lemma from the following
estimations.  For $j$ large enough, the solutions of the above systems can
be approximated by the solutions of the following differential
equation:\beq\label{ddotJ}dv/dt=2\beta r^{\beta-1}(t)\,\,\,\,\,\text{ and }\,\,\,\,\,\,dr/dt=m^2 v(t).\eeq

 Let $k'$ be uniquely defined by
$r_{k'+1}\geq r_{k'}$. First we consider the interval $1\leq j\leq
k'$, i.e.\ where $\{r_j\}$ is decreasing. Note that both $\{r_j\}$
and $\{v_j\}$ are decreasing sequences of positive numbers for
$j=1,\ldots,k'$. Using the equation (\ref{ddotJ}), one can show that for any $j=1,\ldots,k'$ the following relation is  always true: \beq \label{xwbeta}m^2v_j^2 \sim 4 r_j^{\beta}\eeq
and
$$r_j^{\frac{\beta}{2}-1}\sim\left((\beta-2)jm+r_0^{1-\beta/2}\right)^{-1}\sim \left((\beta-2)jm+m^{\frac{\beta-2}{2(\beta-1)}}\right)^{-1},$$
where $|r_0|> m^{-\frac{1}{\beta-1}}$ is the initial position for this sequence.

\begin{lemma}
Let $k''\in [1,k']$ be uniquely defined by the condition \beq
\label{k''}
	 v_{k''-1}>2v_{k'}>v_{k''}.
\eeq Then for all $k'< j \leq k''$ we have  \beq \label{xwbeta3}
  r_j\sim m^2(k'-j)v_{k'}
\eeq
And \beq\label{k''k'}k'-k''\sim \frac{1}{m^{\frac{2(\beta-1)}{\beta}}v_{k'}^{1-\frac{2}{\beta}}}\sim \frac{1}{m^{\frac{2}{\beta}} r_{k'}^{1-\frac{2}{\beta}}}\eeq
\end{lemma}
\proof Due to (\ref{wxm1}),
for any $j\in
[k'',k')$ we have
$$m^2 v_{k'}\leq r_j-r_{j+1}\leq 2m^2 v_j\leq 2m^2 v_{k''}\leq 4m^2 v_{k'}$$
This implies
\beq\label{kk''l}m^2(k'-j)v_{k'}\leq r_j\leq 4m^2(k'-j+1)v_{k'}.\eeq
 Now we get (\ref{xwbeta3}) and $r_{k'}\sim m^2 v_{k'}$.
(\ref{wxm1}) implies that
$$ m^{2(\beta-1)} (k'-j)^{\beta-1} v_{k'}^{\beta-1}\leq  v_{j}-v_{j+1}\leq 3\beta r_{j}^{\beta-1}\leq 3\beta 4^{\beta-1} m^{2(\beta-1)} (k'-j)^{\beta-1} v_{k'}^{\beta-1}$$
therefore
$$
	   v_j \sim v_{k'}+m^{2(\beta-1)} (k'-j)^{\beta} v_{k'}^{\beta-1}
$$
Substituting $j=k''$, then $j=k''-1$ and using (\ref{k''}) we get
$$m^{2(\beta-1)} (k'-j)^{\beta} v_{k'}^{\beta-1}
\sim v_{k'}$$
Taking $j=k''$ in (\ref{kk''l}), we get  (\ref{k''k'})\qed\medskip

We note that the above lemma implies
\beq \label{wxn''}
	r_{k''}^{\beta} \sim m^2 v_{k'}^2,\,\,\,r_{k'}\sim m^2 v_{k'}
\eeq
hence $r_{k''} \ll \varepsilon_m$ and thus $k'' \gg 1$. Next we
consider the case $1< j\leq k''$.

\begin{lemma}\label{lem21} For all $1< j \leq k''$ we have  \beq \label{xwbeta1}
   r_{j}  \sim	 (mj)^{\frac{2}{2-\beta}}.
\eeq
 and
\beq \label{n''n'}
   mk'' \sim (mv_{k'})^{\frac{2-\beta}{\beta}}.
\eeq
 Furthermore, \beq \label{nnn}
   k'' \sim k'\sim k\qquad\text{and}\qquad
 r_{k'} \sim \frac{1}{m^{\frac{2}{2-\beta}} k^{\frac{\beta}{\beta-2}}}
   \eeq
\end{lemma}

\proof Denote $z_j = r_j^{\frac{\beta-2}{2}}$. Then (\ref{wxm1})
and the mean value theorem imply
\begin{align*}
   z_j - z_{j+1}
   &\sim m^{\frac{\beta-2}{2}} J_j^{\frac{\beta-4}{2}}(J_j-J_{j+1})\\
   &\sim m^{\frac{\beta-2}{2}} J_j^{\frac{\beta-4}{2}}(2 v_j) \\
   &\sim z_j^2 \end{align*}
(we used the  relation in (\ref{xwbeta})). Now let $Z_j =
1/z_j$, then
$$
   Z_{j+1} - Z_j \sim Z_{j+1}/Z_j
$$
Note that by (\ref{wxm1}),
$$r_{j}-r_{j+1} \sim r_j^{\beta/2} \ll r_j$$
Hence $Z_{j+1}- Z_j\sim 1$. Since $r_0 \geq \varepsilon_m$,
\beq \label{Z0}
	 Z_0 \leq \varepsilon_m^{-\frac{\beta-2}{2}}=\,\text{const.}
m^{\frac{\beta-2}{2(\beta-1)}},
\eeq
and we obtain
\beq \label{Zm}
   Z_j \sim mj
   \qquad\text{and}\qquad
   z_j \sim (mj)^{-1},
\eeq
which proves  (\ref{xwbeta1}). Now
(\ref{n''n'}) is immediate due to (\ref{wxn''}).
Equation (\ref{n''n'}) also imply $k'' \sim
k'\sim k$ and $mk' \sim (mv_{k'})^{\frac{2-\beta}{\beta}}$. \qed
\begin{lemma}
For all $1\leq j \leq k''$ we have
\beq \label{xmain}
  r_j^{\beta-2} \geq D_1\left(
   j + C_1 \ln j + C_2 j
   \Bigl( \frac jk \Bigr )^{\frac{2\beta}{\beta-2}} + C_3\right)^{-2}
\eeq
where $D_1=\frac{1}{ m^2(\beta-2)^2}$ and $C_1,C_2,C_3>0$ are some
constants.
\end{lemma}

By (\ref{cKrm}) we have proved the following estimation on $\frac{2\cK(r_j)}{\cos\varphi_{j}}$.
\begin{corollary}
For all $1\leq j \leq k''$ we have \beq \label{cKmain}
   \frac{2\tau(r_j,\varphi_j)\cK(r_j)}{\cos\varphi_{j}} \geq
   D \biggl[ j + C_1' \ln j+ C_2'  j
   \Bigl( \frac jk \Bigr )^{\frac{2\beta}{\beta-2}}
   + C_3'\biggr]^{-2}
\eeq where $D=\frac{2\beta(\beta-1)}{ (\beta-2)^2}$.
\end{corollary}

Note that $\tau(r_j,\varphi_j)\cK(r_j)/\cos\varphi_j\sim j^{-2}$, which does not depend on $m$.
Now we use the relation (\ref{cBXm}) and get the estimation for $\tau(x_j)\cB(x_{j})$.
\begin{lemma} \label{LmB}
For all $1\leq j < k''$ we have \beq \label{cBmain2}
	\tau(x_j)\cB(x_{j})\geq
   \frac{ A}{ j} + \frac{C_1'\ln j}{j^2},
\eeq where $A>0$ satisfies $ A^2 -  A = \frac{2\beta(\beta-1)}{
(\beta-2)^2}$, hence $A=\frac{2\beta-2}{\beta-2}$.
\end{lemma}

Now we are ready to estimate the expansion factor $\Lambda_{m,k}(x)$
given by (\ref{Lambdan21}).

\begin{lemma}
 \beq \label{Lambdan''}
	 \prod_{j=1
	 }^{k''} \bigl(1 + \tau(x_j) \cB(x_j)\bigr)
	 \geq C k^{\frac{2(\beta-1)}{\beta-2}} \eeq where $C>0$ is a constant.
\end{lemma}

\proof Note that $\tau(x_j) > 2m$. Hence, due (\ref{cBmain2}), we
have
$$
	 \ln \Biggl[\prod_{j=1}^{k''-1} \bigl(1 + \tau(x_j)
	 \cB(x_j)\bigr)\Biggr]
	 >\sum_{j=0}^{k''}
   (\frac{ A}{ j} + \frac{C_1'\ln j}{j^2})
$$
with some large constant $C_1'>0$. Therefore,
$$
	 \ln \Biggl[\prod_{j=1}^{k''} \bigl(1 + \tau(x_j)
	 \cB(x_j)\bigr)\Biggr]
	 >A\ln k'' + \text{const} >
	 A\ln k +\text{const},
$$
where the last inequality follows from (\ref{nnn}). Lastly, note
that $A =\frac{2(\beta-1)}{\beta-2}$, which completes the proof
of the lemma. \qed
\medskip

Next we are ready to estimate the expansion factor $\Lambda_{m,k}(x)$.

\begin{lemma}
\beq \label{Lambdan2}
	 \Lambda_{m,k}(x)=
	 \prod_{j=1}^{k-1} \bigl(1 + \tau(x_j) \cB(x_j)\bigr)
	 \geq C   k^{\frac{4(\beta-1)}{\beta-2}-1}
\eeq where $C>0$ is a constant.
\end{lemma}

\proof We will use  the time-reversibility of the billiard
dynamics. Let $V^u$ and $V^s$ be two unit vectors (in the p-norm)
tangent to the unstable and stable manifolds, respectively, at the
point $x$.  Since $\cB(x_1)^{-}=\cO(m^{-1})$, (\ref{cBt}) implies that the slope of the vector
$V^u_{k'}$ is
$$
   \frac{d\varphi}{dr} = \cos\varphi_{1}\,\,\cB(x_{1})^- +
   \cK(r_{1})\leq  c \frac{1}{m^{2}}.
$$ Hence the vector $V^u$ makes an angle
of order $c m^{-2}$ with the horizontal $r$-axis. By the
time reversibility, the vector $V^s$ makes an angle less than $
-c m^{-2}$ with the horizontal $r$-axis. Thus the area of the
parallelogram $\Pi$ spanned by $V^u$ and $V^s$ is of order $ \cO(m^{-2})$.

Consider the parallelogram $\Pi_{k'} = D_x\cF^{k'm} (\Pi)$ spanned
by the vectors $V_{k'}^u = D_x\cF^{k'm} (V^u)$ and $V_{k'}^s =
D_x\cF^{k'm} (V^s)$. Since the map $\cF^{k'm}$ preserves the
measure $d\mu = \cos\varphi\, dr\, d\varphi$, we have
$$
   \cos\varphi_{k'} \text{Area}(\Pi_{k'}) =
   \cos\varphi_1\,\, \text{Area}(\Pi).
$$
Note that $\cos\varphi_1 \leq c m^{-1}$ and $\cos\varphi_{k'} \approx m^{-1}$, hence
$$
  \text{Area}(\Pi_{k'}) \leq \text{Area}(\Pi)\frac{\cos\varphi_1}{\cos\varphi_k'}\leq c m^{-2}.
$$
On the other hand,
$$
   \text{Area}(\Pi_{k'}) = |V^u_{k'}|_p\,|V^s_{k'}|_p\,\sin\gamma_{k'}
$$
where $ |V^u_{k'}|_p$ and $|V^s_{k'}|_p$ denote the lengths of these
vectors in the p-norm  and $\gamma_{k'}$
denotes the angle between them.

Next we estimate $\gamma_{k'}$. It easily follows from
(\ref{cBmain2}) that
$$\cB(x_{k'})\sim \left(km+\frac{1}{\cB(x_0)}\right)^{-1}\sim \frac{1}{km}$$
Now by (\ref{k''k'}) and (\ref{nnn}) we have  \beq\label{cBxk}
   \cos\varphi_{k'}\cB^-(x_{k'})\geq c_2 \frac{1}{m^2k},\,\,\,\,\,\,\, \cK(r_{k'})\geq c_3 \frac{1}{m^2k^{\beta}}
\eeq So  (\ref{cBXm}) implies that
$$
	\frac{d\varphi}{dr} \geq \frac{C}{m^2 k}
$$
for some constant $C>0$. Hence  $\sin\gamma_{k'} \geq \frac{C}{m^2 k}$ for
some constant $c_1>0$, and we obtain
$$
   |V^u_{k'}|_p\,|V^s_{k'}|_p\leq c_2\,\frac{m^2k}{m^{2}}=c_2\,k
$$
for some constant $c_2>0$.  Note that
$$
  |V^u_{k'}|_p = \Lambda_{m,k}^{(1)}(x) \,  |V^u|_p
  \sim \Lambda_{m,k}^{(1)}(x).
$$
where $\Lambda_{m,k}^{(1)}(x)$ is the expansion factor  corresponding to collisions from $1$ to $k'$, with
    $$\Lambda_{m,k}^{(1)}(x):= \prod_{j=1}^{k'} \bigl(1 + \tau(x_j) \cB(x_j)\bigr)$$ By the time reversibility of the
billiard dynamics, the contraction of stable vectors during the
time interval $(0,k')$ is the same as the expansion of the
corresponding unstable vectors during the time interval $(k',k-1]$,
hence
$$
  |V^s_{k'}|_p \sim \Lambda_{m,k}^{(2)}(x)^{-1} |V^s|_p
  \sim  \Lambda_{m,k}^{(2)}(x)^{-1},
$$where $\Lambda_{m,k}^{(2)}(x)$ is the expansion factor  corresponding to collisions from $k'+1$ to $k-1$, with
    $$\Lambda_{m,k}^{(2)}(x):= \prod_{j=k'+1}^{k-1} \bigl(1 + \tau(x_j) \cB(x_j)\bigr)$$
Therefore, \beq \label{Lambda12}
   \Lambda_{m,k}^{(2)}(x) > \frac{ \Lambda_{m,k}^{(1)}(x)}{c k}
         \eeq for some constant $c>0$. Now we combine   with
(\ref{Lambdan''})  and get
$$
  \Lambda_{m,k}(x) \geq \Lambda_{m,k'}^{(1)}(x)\Lambda_{m,k'}^{(2)}(x)\,
  \geq C  k^{\frac{4(\beta-1)}{\beta-2}-1}.
$$
This proves Lemma~\ref{Lambdan2} for $W\subset C_{m,k}$.

We finally consider the remaining case $W \subset C'_{m,k}$. In that
case $k'$ can be defined as the turning point, i.e.\ by
$r_{k'}<r_{k'-1}$ and $r_{k'}<r_{k'+1}$. Observe that if
$x'=(r',\varphi')\in C'_{m,k}$, then there exists another point
$x=(r,\varphi)\in C_{m,k}$ with $r=r'$ and $\varphi< \varphi'$,
whose trajectory goes through the window. Since $\varphi'<
\varphi$, it follows that the $r$-coordinate of the point
$\cF^{jm}(x)$ will be always smaller than the $r$-coordinate of
the point $\cF^{jm}(x')$, for all $1\leq j\leq k-2$. This
observation and the bound (\ref{xmain}) that we have proved for
$J_j$ implies that the same bound holds for $r_k$ and for all
$1\leq j\leq k''$. The rest of the proof of
Lemma~\ref{Lambdan2} for $x'\in C_{m,k}'$ is identical to that
of the case $x\in C_{m,k}$.
\qed
\\

Now we prove Proposition \ref{PrAux1} using (\ref{Lambdan21}). Combining with the above lemma together with (\ref{Lambdan21}), we have shown
        \beq \label{Lambdan3}
	 \Lambda(x)=
	 \prod_{j=0}^{k-1} \bigl(1 + \tau(x_j) \cB(x_j)\bigr) \frac{\cos\varphi_j}{\cos\varphi_{j+1}}
	 \geq C  m^{1+\frac{1}{\beta-1}} k^{\frac{4(\beta-1)}{\beta-2}-1}
\eeq where $C>0$ is a constant.

When $W$ is a vertical curve, since by the differential (\ref{DTdiff}), $dr_1/d\varphi=\tau/\cos\varphi_1$, its expansion factor for the collision under $\cF^m$ does not depend on $\cK$.  This implies that the expansion factor along $x\in W$ satisfies   \beq\label{verticalLb}
  \Lambda(x) = \prod_{j=0}^{k-1} \bigl(1 + \tau(x_j)
  \cB(x_j)\bigr)\frac{\cos\varphi_j}{\cos\varphi_{j+1}}\geq \frac{C\tau(x)}{\cos\varphi_k} \prod_{j=1}^{k-1} \bigl(1 + \tau(x_j)
  \cB(x_j)\bigr)\geq C  m^{2} k^{\frac{4(\beta-1)}{\beta-2}-1}\eeq
  Proposition \ref{PrAux1} is now proved.

\medskip\noindent\textbf{Acknowledgement}.
This paper is written in memory of   Professor Nikolai Chernov.
The author  is also partially supported by NSF CAREER Grant  (DMS-1151762) and the Simons Fellowship. The author also would like to thank Dmitry Dolgopyat for many helpful discussions on the maps considered in this paper.

\end{document}